\documentclass[10pt]{article}
\usepackage{latexsym}
\usepackage{amssymb}
\usepackage{amsmath}
 \usepackage{array}
 \usepackage{enumerate}
\usepackage{amsfonts}
\usepackage{amsthm}
\newcommand{\N}{\mathbb{N}}

\newcommand{\R}{\mathbb{R}}
\newcommand{\C}{\mathbb{C}}

\newcommand{\kommentar}[1]{}
\def\Sc{\mathop{\rm Sc}\nolimits}

\newtheorem{theorem}{Theorem}
\newtheorem{definition}{Definition}

\newtheorem{lemma}{Lemma}

\newtheorem{remark}{Remark}
\newenvironment{beweis}{\begin{proof}[Proof]}{\end{proof}}

\title{The Szeg\"o metric associated to Hardy spaces of Clifford algebra valued functions and some geometric properties}

\author{Dennis Grob
\thanks{Lehrstuhl A f\"ur Mathematik, Rheinisch-Westf\"alische Technische Hochschule Aachen,
D-52056 Aachen, Germany. E-mail: {\tt
dennis.grob@mathA.rwth-aachen.de}} \and Rolf
S\"oren~Krau{\ss}har
\thanks{Fachbereich Mathematik, Technische Universit\"at Darmstadt,
 Schlo{\ss}gartenstr. 7, 64289 Darmstadt, Germany. E-mail: {\tt krausshar@mathematik.tu-darmstadt.de} } }

\begin{document}

\maketitle

\begin{abstract}
In analogy to complex function theory we introduce a Szeg\"o
metric in the context of hypercomplex function theory dealing with
functions that take values in a Clifford algebra. In particular,
we are dealing with Clifford algebra valued functions that are
annihilated by the Euclidean Dirac operator in $\mathbb{R}^{m+1}$.
These are often called monogenic functions. As a consequence of
the isometry between two Hardy spaces of monogenic functions on
domains that are related to each other by a conformal map, the
generalized Szeg\"o metric turns out to have a pseudo-invariance
under M\"obius transformations. This property is crucially applied
to show that the curvature of this metric is always negative on
bounded domains. Furthermore, it allows us to establish that this
metric is complete on bounded domains.
\end{abstract}

\section{Introduction and basic notions}

The theory of functions with values in a Clifford algebra that are annihilated by the Dirac operator represents one possibility to generalize classical complex function theory to higher dimensions and can be regarded as a counterpart to function theory of several complex variables. In the classical case dealing with functions in one complex variable and in the context of several complex variables theory, the study of the properties of the Bergman metric and of the related Szeg\"o metric has been a subject of interest over the past decades. These metrics turned out to be a useful tool in differential geometry, too. More concretely, using special properties of the Szeg\"o metric it was possible to establish a close connection between the completeness of the Bergman metric and the smoothness of the boundary. See for instance works of Diederich, Forn\ae{}ss, Herbort, Kobayashi, Krantz, Pflug, and others ( \cite{Died70,Koba59,Koba62,JPZ00,BP98,Herb99,DFH84,Kran01}) in the context of function theory of several complex variables. It is of great interest to investigate on which domains the Szeg\"o metric is complete. Recent years publications like \cite{Herb99,JPZ00,KK03} show that there is still both interest and need for research in the field of several complex variables, also in the study basic properties of these metrics (\cite{ForLee2009}). Publications for example by Chen \cite{Chen04,Chen03,Chen99}, Nikolov \cite{Niko03} and Herbort \cite{Herb99,Herb04} examined the Bergman completeness on different kinds of manifolds and domains with varying smoothness of the boundary, while Blumberg, Nikolov, Pflug and others studied estimates and behaviour of both the Bergman kernel and the Bergman metric on different types of domains, see for example \cite{Blum05,NP03,NP03_2}.
\par\medskip\par

The strong results in the context of one and several complex variables theory provide a motivation to investigate whether one can establish similar results in the context of Clifford algebra valued functions. In order to start we  first need to briefly recall some basic terms and notions.

\par\medskip\par

Let $\{e_1,e_2,\ldots,e_m\}$ be the standard basis of the
Euclidean vector space ${\mathbb{R}}^m$. Further, let ${\mathcal{C}l}_{0m}$
be the associated real Clifford algebra  in which
$$e_i e_j + e_j e_i = - 2 \delta_{ij} e_0,\quad i,j=1,\cdots,m,$$
holds. Here $\delta_{ij}$ is Kronecker symbol. Each
element $a \in {\mathcal{C}l}_{0m}$ can be represented in the form $ a =
\sum_A a_A e_A$ with $a_A \in {\mathbb{R}}$, $A \subseteq
\{1,\cdots,m\}$, $e_A = e_{l_1} e_{l_2} \cdots e_{l_r}$, where $1
\le l_1 < \cdots < l_r \le m,\;\; e_{\emptyset} = e_0 = 1$. The
scalar part of $a$, denoted by $\Sc(a)$, is defined as the $a_0$
term. The Clifford conjugate of $a$ is defined by $\overline{a} =
\sum_A a_A \overline{e}_A$, where $\overline{e}_A =
\overline{e}_{l_r} \overline{e}_{l_{r-1}} \cdots
\overline{e}_{l_1}$ and $\overline{e}_j = - e_j$ for
$j=1,\cdots,m,\; \overline{e}_0 = e_0 = 1$.
For the space of paravectors ${\R} \oplus {\R}^n \subset {\mathcal{C}l}_{0m}$, which consists of elements of the form $z = z_0 + z_1 e_1 + z_2 e_2 + \cdots +
z_m e_m$, we also write $\R^{m+1}$ for simplicity. For such an element $z$, the term $\sum_{i=1}^m z_i e_i$ is called the vector part of $z$ and is denoted by $\overrightarrow{z}$ or $Vec(z)$. The standard scalar product
between two Clifford numbers $a,b \in {\mathcal{C}l}_{0m}$ is defined by
$\langle a,b \rangle := \Sc(a\overline{b})$. This induces a pseudo
norm on the Clifford algebra, viz $\|a\| = ( \sum\limits_A
|a_A|^2)^{1/2}$.

\par\medskip\par
In the Euclidean flat space $\mathbb{R}^{m+1}$, the associated Dirac operator has the simple form
\begin{equation}
\label{cr} D_z = \frac{\partial }{\partial z_0} +
\sum\limits_{j=1}^m \frac{\partial }{\partial z_j} e_j \quad \quad
z_j \in {\mathbb{R}} \;\;j=0,\ldots,m.
\end{equation}
In this particular context it is often called the generalized Cauchy-Riemann operator.
\par\medskip\par
Next let $G \subseteq {\mathbb{R}}^{m+1}$ be an open set. A real
differentiable function $f: G \rightarrow \mathcal{C}l_{0m}$  that
satisfies inside of $G$ the system $D_z f = 0$ (or $ f D_z
=0$) is called left (right) monogenic with respect
to the paravector variable $z$, respectively. In the
two-dimensional case (i.e. $m=1$) this differential operator coincides with the
classical complex Cauchy-Riemann operator. In this sense, the set
of monogenic functions can also be regarded as a higher dimensional
generalization of the set of complex-analytic functions. For
details, see for example \cite{DSS} and elsewhere. The first order
operator $D_z$ factorizes the Euclidean Laplacian, viz
$D_z\overline D_z=\Delta_z$.

\par\medskip\par

An important example of a separable Hilbert space of Clifford
valued monogenic functions that satisfies in particular the
Bergman condition $\|f(z)\|\leq C(z)\|f\|_{L^2}$ is the closure of
the set of functions that are monogenic inside a domain of
${\mathbb{R}}^{m+1}$ and square integrable on the boundary. This
space is called the Hardy space of monogenic functions. Each Hardy
space of monogenic Clifford algebra valued functions is a Banach
space which is endowed with the Clifford-valued inner product $
(f,g) = \int_{\partial G} \overline{f(z)} g(z) dS$, where $f$ and
$g$ denote
elements of the Hardy space. Basic contributions to the study  of these function spaces have already been conducted in the 1970s, see for example \cite{BD78}.\\
The reproducing kernel $K_G$ of the Hardy space of monogenic functions, called the Szeg\"o
kernel, is uniquely defined and satisfies $f(z) = \int_{\partial G}
{K_{G}(z,w)} f(w) dV $ for any function $f$ that belongs to that Hardy space.
The function $K_G(z,z)$ represents a non-negative real valued function (see, for example, \cite{BDS,GHS} or below). This in turn can be used to define a Hermitian metric on the domain $G$. This metric then is called the Szeg\"o metric.
\par\medskip\par

A reason why we focus on this paper on the study of the Szeg\"o metric and not on a generalization of the Bergman metric is that we are interested in metrics which are invariant under conformal transformations. Although we do not get a genuine total invariance, the Szeg\"o metric turns out to exhibit an invariance up to a simple scaling factor. This is a consequence of the well-known transformation formula between two Hardy spaces of monogenic functions, see for instance \cite{Cno02,Cno94}. This transformation formula expresses the isometry of two Hardy spaces of monogenic functions in the case that both domains can be transformed into each other by a conformal map. As a consequence of Liouville's theorem, this is exactly the case when both domains can be transformed into each other by a M\"obius transformation. A M\"obius transformation is generated by  reflections at spheres and hyperplanes. Following classical works like for instance,\cite{Ahlf86}, we recall that a M\"obius transformation in $\mathbb{R}^{m+1}$ can be expressed in the simple form $T(z) = (az+b)(cz+d)^{-1}$ when $a,b,c,d$ are Clifford numbers that satisfy the constraints
\begin{enumerate}
\item[(i)] $a,b,c,d$ are products of paravectors from
$\mathbb{R}^{m+1}$
\item[(ii)] \ $\tilde{H} := a\tilde{d} - b\tilde{c} = 1 $
\item[(iii)] \ $ac^{-1}, \, c^{-1}d \in \R^{m+1} $ \, for \, $ c \neq 0 \, $ and  \, $ bd^{-1} \in \R^{m+1} $ \, for \,$ c=0 $ .
\end{enumerate}
\par\medskip\par
Recently some attempts have been made to introduce a Bergman
metric in the context of hypercomplex function theory, see e.g.
\cite{Cerv09}. In the very interesting work \cite{Cerv09} the
author studied generalizations of the famous Bergman
transformation formula. However, the transformation formula
presented in \cite{Cerv09} does not represent an isometry between
the two Bergman spaces. We do not have an isometry as a
consequence of the appearance of the extra factor
$\frac{1}{|z|^2}$. That factor is only equal to $1$ on the
boundary of the unit ball and not in the interior of it. Due to
the non-existence of an isometry of two Bergman spaces of two
conformally equivalent domains it seems very unlikely that one can
establish a total analogy of the classical transformation formula
for the Bergman kernel when dealing with Clifford algebras that
are different from $\C$. This is one motivation why we look in
this paper at the hypercomplex Szeg\"o kernel instead.
\par\medskip\par
The existence of an isometric transformation formula for the hypercomplex Szeg\"o kernel is a fundamental advantage.
In Section~2 of this paper we depart from that transformation formula in order to establish the invariance of the hypercomplex Szeg\"o metric (up to a scaling factor) under conformal transformations. This invariance in turn is crucially applied in Section~3 to show that the curvature of this metric is negative on bounded domains. Finally, in Section~4, this property again is used to establish the completeness of this metric on bounded domains. These results can be regarded as a close analogy to the results from one and several complex variables theory.

\section{Simple consequence of the transformation formula}

The starting point of our consideration is the following transformation formula of the Szeg\"o kernel, cf. for example \cite{Cno94,Krau98,Cno02}:

\begin{theorem}
Let $G, G^* \, \subseteq \R^{m+1}$ be two bounded domains with smooth boundaries. Let $T: G \longrightarrow G^*$ be a M\"obius transformation with coefficients $a,b,c,d$, which maps $G$ conformally onto $G^*$. By $K_{G}$ and $K_{G^*}$ we denote the Szeg\"o kernels belonging to $H^2(\partial G,\mathcal{C}l_{0m})$ resp. $H^2(\partial G^*,\mathcal{C}l_{0m})$.
We further write $z^*:=T(z)$ for $z \in G$. Then we have the following transformation formula:
\[
   K_{G}(z,\zeta)= \frac{\overline{cz+d}}{\vert cz+d\vert^{m+1}} K_{G^*}(z^*,\zeta^*) \frac{c\zeta + d}{\vert c\zeta+d \vert^{m+1}}.
\]
\end{theorem}

To proceed we follow the theory of \cite{Kran04} and introduce in the same spirit

\begin{definition}
\begin{enumerate}
 \item  A Hermitian metric on a domain $G$ is a continuous function $\lambda$ on $G$, which is positive up to isolated zeros.
\item The length of a piecewise continuously differentiable path $C$ in $G$ with parameterisation $\gamma : [a,b] \to G$ with respect to a Hermitian metric $\lambda$ is defined by
\[
  L_{\lambda}(C) = \int_{C} \lambda(z) \vert dz \vert.
\]
 From now on, we denote Hermitian metrics by
\[
  ds = \lambda(z) |dz|
\]
and call $ds$ the metric's \textit{line element}.
\end{enumerate}
\end{definition}

\begin{definition}
The distance function $d: G \times G \to \R$ that is associated to a Hermitian metric $ds$ is defined by
\[
 d(z_1, z_2) = \inf \left\lbrace L_{ds}(C) : C \text{ path in G from } z_1 \text{ to } z_2 \right\rbrace.
\]
\end{definition}

The proof that $d$ is a distance function is analogue to the complex case, which is described in \cite{FL}.

\begin{definition}
 A Hermitian metric $ds = \lambda(z) |dz|$ is called \textit{regular}, if $\lambda$ is at least twice continuously differentialble and does not have any zeros.
\end{definition}
Finally, one introduces the Gau{\ss}ian curvature as usual by
\begin{definition}
 Let $ds = \lambda(z) |dz|$ be a regular metric. Then the function
\[
   \mathfrak{G}_{ds}(z):=-\frac{1}{\lambda(z)^2} \Delta \log \lambda(z)
\]
is called the (Gau\ss ian) curvature of the metric $ds$.
\end{definition}
This naturally leads to the following definition of the Szeg\"o metric in the context of Hardy spaces of monogenic functions:
\begin{definition}
Let $G \subset \R^{m+1}$ be a domain and $K_G$ the Szeg\"o kernel of this domain. The Hermitian metric, which is defined by
\[
   ds_G := K_G (z,z) \vert dz \vert,
\]
is called the \it{Szeg\"o metric}.
\end{definition}
To establish that $ds_G$ actually is a metric we only have to
verify that $K_G (z,z)$ is a positive real valued expression. Since $K_G$ is the reproducing kernel of the Hardy space, we can conclude
\[
  K_G(z,z) = (K_G(z, \cdot), K_G(z, \cdot)) = \Vert K(z, \cdot) \Vert,
\]
where $\Vert \cdot \Vert$ is the usual $L_2$-norm induced by the inner product on $L^2(G)$. $K_G(z,\cdot)$ cannot be the zero function, because otherwise the Hardy space would be degenerated. Therefore, $K_G(z,z)$ is real valued and positive for any $z$. Thus, $ds_G$ is indeed a well-defined metric.
\par\medskip\par
Our starting point is to prove.
\begin{theorem}
The Szeg\"o metric is invariant under M\"obius transformations up to the automorphy factor
\[
\vert cz + d \vert^{2m-2}.
\]
\end{theorem}
\begin{proof}
Let $T$ be a M\"obius transformation, expressed through $\left( \begin{array}{lr} a & b \\ c & d \end{array} \right) $ with Clifford numbers $a, b, c, d$ satisfying the usual constraints as mentioned above.\\
Then
\begin{eqnarray*}
  ds_{T(G)} (z) & = & K_{T(G)}(T(z),T(z)) \vert dT(z) \vert \\
  & = & K_{T(G)}(T(z),T(z)) \frac{1}{\vert cz + d \vert^{2}} \vert dz \vert \\
  & = & \vert cz + d \vert^{2m-2} \underbrace{K_{T(G)}(T(z),T(z))}_{\in \R} \frac{\overline{cz + d}}{\vert cz + d \vert^{m+1}} \frac{cz + d}{\vert cz + d \vert^{m+1}} \vert dz \vert\\
  & = & \vert cz + d \vert^{2m-2} \frac{\overline{cz + d}}{\vert cz + d \vert^{m+1}} K_{T(G)}(T(z),T(z)) \frac{cz + d}{\vert cz + d \vert^{m+1}} \vert dz \vert \\
  & = & \vert cz + d \vert^{2m-2} K_G(z,z) \vert dz \vert,
\end{eqnarray*}
where we apply the transformation formula for the Szeg\"o kernel
in the last line.
\end{proof}

\begin{remark}
 In the complex case, i.e. $m=1$, the automorphic factor simplifies to $1$ and the metric thus is completely  invariant under M\"obius transformations.
\end{remark}

For the sake of readability and better comparability to the complex case, we introduce the following notation:\\
For a continuously differentiable function $f$ we define
\[
  f' := \bar{D} f .
\]
Here, an in all that follows, $D$ stands for $D_z :=\frac{\partial }{\partial x_0} + \sum\limits_{i=1}^m e_i \frac{\partial }{\partial x_i}$. We will leave out the index when the variable of differentiation is unambiguous.\\
Next we need the following generalized product role:
\begin{theorem}\label{produktregel}
Let $f$ and $g \in C^1(\Omega, \mathcal{C}l_{0,m})$. Then
\[
  D(fg) = (Df)g + \bar{f}(Dg) + 2R(f)(\partial_0 g) + 2 \sum_{i=1}^m \sum_{ \begin{array}{c} \scriptstyle A \subseteq \{1,\ldots, m\}, \\ \scriptstyle  i \notin A : |A| \equiv_4 0, 1 \\ \scriptstyle i \in A : |A| \equiv_4 2, 3 \end{array}  } f_A (\partial_i g)
\]
and
\[
  (fg)D = (fD)\bar{g} + f(gD) + 2(f \partial_0)R(g) +  2 \sum_{i=1}^m \sum_{\begin{array}{c} \scriptstyle A \subseteq \{1,\ldots, m\}, \\ \scriptstyle  i \notin A : |A| \equiv_4 0, 1 \\ \scriptstyle i \in A : |A| \equiv_4 2, 3 \end{array}} g_A (\partial_i f),
\]
as well as
\[
  \overline{D}(fg) = (\overline{D}f)g + \bar{f}(\overline{D}g) + 2R(f)(\partial_0 g) - 2\sum_{i=1}^m \sum_{\begin{array}{c} \scriptstyle A \subseteq \{1,\ldots, m\}, \\ \scriptstyle  i \notin A : |A| \equiv_4 0, 1 \\ \scriptstyle i \in A : |A| \equiv_4 2, 3 \end{array}} f_A (\partial_i g)
\]
and
\[
  (fg)\overline{D} = (f\overline{D})\bar{g} + f(g\overline{D}) + 2(f \partial_0)R(g) -  2 \sum_{i=1}^m \sum_{\begin{array}{c} \scriptstyle A \subseteq \{1,\ldots, m\}, \\ \scriptstyle  i \notin A : |A| \equiv_4 0, 1 \\ \scriptstyle i \in A : |A| \equiv_4 2, 3 \end{array}} g_A (\partial_i f),
\]
where we use the notation $R(z):=z - \Sc(z)$ for a Clifford number $z$.
\end{theorem}
\begin{proof}
The proof follows almost along the same lines as the proof for the case of quaternionic vectors, which can be found in \cite{GS90}.  Therefore, we will only give the additional, necessary arguments for the proof of the first equation (the other three cases can then be treated analogously).\\
By applying simple calculations, we obtain that 
\begin{eqnarray*}
 \partial_0 (fg) &=&  (\partial_0 f)g + f(\partial_0 g) \\
&=& (\partial_0 f)g + \overline{f}(\partial_0 g) + 2\;R(f)(\partial_0 g).
\end{eqnarray*}
Furthermore, notice that
\begin{equation*}
 e_i e_A = \left\lbrace \begin{array}{l}
            (-1)^{|A|} e_A e_i \; \text{for} \; i \notin A  \\
            (-1)^{|A|-1} e_A e_i  \; \text{for} \; i \in A  \\
           \end{array} \right.
\end{equation*}
holds for any $i \in \{1,\ldots , m \}$ and any $A \subseteq \{1,\ldots , m \}$.\\
In view of 
\begin{equation*}
 \overline{e_A} = \left\lbrace \begin{array}{l}
             e_A \; \text{for} \; |A| \equiv_4 0, 3  \\
             -e_A \; \text{for} \; |A| \equiv_4 1, 2 \\
           \end{array} \right. ,
\end{equation*}
we obtain $e_i e_A = \overline{e_A} e_i$ for all choices of $A$ except for $i \notin A , |A| \equiv_4 0, 1$ and $i \in A , |A| \equiv_4 2, 3$. For details concerning the calculations with imaginary units see also \cite{BDS}.
\end{proof}


 From now on, we denote by
\[
  \partial^z_i, \qquad i \in \{0, \ldots, m \}
\]
the real partial derivative with respect to the $i$-th component and the variable $z$.

\section{The curvature of the Szeg\"o metric}

The main aim of this section is to show that the Szeg\"o metric is
negative on bounded domains. This provides us with a generalization of the well-known analogous statement in classical complex analysis, as presented for instance in \cite{FL}. This result will then further be applied in the following section where we prove that the Szeg\"o metric is complete. 

To proceed in this direction we first need to prove the following

\begin{lemma}\label{drei}
 Let $G$ be a domain and $\left( \phi_k(z) \right)_{k \in \N}$ be an orthonormal basis of the Hardy space over $G$. Then
\[
  \sum_{k=1}^{\infty} \vert \phi'_k(z) \vert^2 < \infty .
\]
\end{lemma}
\begin{proof}
 Let $K=K_G$ be the Szeg\"o kernel of $G$. Due to the monogenicity resp. anti-monogenicity in the components and the compact convergence of the series expansion, we can infer that
\begin{equation}\label{eins}
 \overline{D} K(z,w) = \sum_{k=1}^{\infty} \phi'_k(z) \overline{\phi_k(w)} < \infty ,
\end{equation}
and the series is again compactly convergent, as a consequence of Weierstra\ss ' convergence theorem .\\
For an arbitrary $k \in \N$ we have
\begin{eqnarray}\label{zwei}
 ( \overline{D_z} + \overline{D_w} ) \phi_k (z) \phi_k (w) (D_z + D_w) &=& \phi'_k (z) \overline{\phi_k (w)}D_z + \overline{D_w} \phi_k (z) \overline{\phi_k (w)} D_z \\
&& + \overline{D_w} \phi_k (z) \overline{\phi'_k (w)} + \phi'_k (z) \phi'_k (w) ,
\end{eqnarray}
as well as (in view of Theorem \ref{produktregel})
\begin{eqnarray*}
 \phi'_k (z) \overline{\phi_k (w)}D_z &=& (\phi'_k(z)D_z) \phi_k (w) + \phi'_k (z) (\underbrace{\overline{\phi_k (w)} D_z}_{=0} ) \\
&&  + 2 ( \partial^z_0 \phi'_k (z) \overline{\phi_k (w)} ) + 2   \sum_{i=1}^m \sum_{\begin{array}{c} \scriptstyle A \subseteq \{1,\ldots, m\}, \\ \scriptstyle  i \notin A : |A| \equiv_4 0, 1 \\ \scriptstyle i \in A : |A| \equiv_4 2, 3 \end{array}} ((\phi'_k (z)) \partial^z_i) (\overline{\phi_k (w)} )_A  \\
\bar{D_w} \phi_k (z) \overline{\phi'_k (w)} &=& \underbrace{(\bar{D}_w \phi_k (z))}_{=0} + \overline{ \phi_k (z) } \underbrace{(\bar{D}_w)  \overline{ \phi'_k (z) } }_{= \overline{\phi'_k (w)D_w}=0} \\
&& + 2 Vec(\phi_k (z)) (\partial_0^w \overline{\phi'_k (w)}) +  2 \sum_{i=1}^m \sum_{A} (\phi_k (z))_A \partial_i^w \overline{ \phi'_k (w) }\\
 \bar{D}_w \phi_k (z) \overline{\phi_k (w)} D_z &=& \left[  (\bar{D}_w \phi_k (z)) \overline{\phi_k (w)} + \overline{\phi_k (z)} \overline{\phi_k (w) D_w} \right. \\
&& \left.  + 2 Vec(\phi_k (z)) (\partial^w_0 \overline{\phi_k (w))} + 2 \sum_{i=1}^m \sum_{A} (\phi_k (z))_A \partial_i^w \overline{\phi_k (w)} \right] D_z\\
&=& \left[ + 2 Vec(\phi_k (z)) (\partial^w_0 \overline{\phi_k (w))} + 2 \sum_{i=1}^m \sum_{A} (\phi_k (z))_A \partial_i^w \overline{\phi_k (w)} \right] D_z,\\
\end{eqnarray*}
where each sum is taken over the same set of $A$s.\\
Because of
\[
 \lim_{w \to z } \sum_{k=1}^{\infty} ( \overline{D_z} + \overline{D_w} ) \phi_k (z) \phi_k (w) (D_z + D_w)  = \overline{D_z} K(z,z) D_z < 0
\]
and in view of (\ref{zwei}) it suffices to show that the sums in the expanded terms of
\begin{eqnarray*}
(\phi'_k(z)D_z) \phi_k (w)\\
\phi'_k (z) \overline{\phi_k (w)}D_z, \\
\overline{D_w} \phi'_k (w) \overline{\phi'_k (w)}, \\
 \overline{D}_w \phi_k (z) \overline{\phi_k (w)} D_z, \\
\end{eqnarray*}
converge in each case.
Both the series expansion of $K$ and of (\ref{eins}) are convergent and well-defined. As $K$ is the Szeg\"o kernel, $K$ is infinitely many times continuously real differentiable, as well as (\ref{eins}). Thus, we are able to  differentiate partially with respect to any component in any order as many times as we want or can apply the Cauchy-Riemann operator on $K$ and (\ref{eins}). As both series converge compactly, we can switch these actions with the summation due to Weierstra\ss ' Theorem. So the summation of the terms above yields to a convergent series in each case.
\end{proof}

 From now on, we will leave out the index and just write $K$ for the Szeg\"o kernel, if it is clear which is the corresponding domain of the kernel.\\
We introduce further notation
\[
  K_z := \overline{D}_z K(z, \cdot)
\]

\[
  K_{\bar{z}} := K(z, \cdot) D_z
\]

\begin{lemma}\label{vier}
 Let $G$ be a bounded domain and $K$ be the Szeg\"o kernel of $G$. Then
\[
  K(z,z)( K(z,z) K_{\bar{z}z}(z,z) - K_z(z,z) K_{\bar{z}}(z,z)) > 0 .
\]
\end{lemma}
\begin{beweis}
First, we show that $K_z(z,\cdot)$ is again an element of the Hardy space.\\
For that matter:\\
As $K$ is monogenic in $z$, so is $K_z$. Furthermore, we have in
view of (\ref{eins}):
\begin{eqnarray*}
  \int_{G} \vert K_z (z, w) \vert^2 dV_w &=& \int_{G} \vert \sum_{k=1}^{\infty} \phi'_k(z) \overline{\phi_k(w)} \vert^2 dV_w  \\
&\leq& \int_{G} \sum_{k=1}^{\infty} \vert \phi'_k(z) \overline{\phi_k(w)} \vert^2 dV_w \\
&\leq & \int_{G} \left( \sum_{k=1}^{\infty} \vert \phi'_k(z) \vert^2 \right) \cdot \underbrace{\left( \sum_{k=1}^{\infty} \vert \overline{\phi_k(w)} \vert^2 \right)}_{=K(w,w)} dV_w \\
&=& \sum_{k=1}^{\infty} \vert \phi'_k(z) \vert^2 \cdot \int_{G}
K(w,w) dV_w .
\end{eqnarray*}
Here, we applied the H\"older inequality and in the last line we
applied Lemma \ref{drei} stating that $\sum_{k=1}^{\infty} \vert
\phi'_k(z) \vert^2 = K_{\bar{z}z} < \infty$.

As $K(z,z)$ is $L^2$-integrable in both components (because $K_z(z,\cdot)$ is an element of the Hardy space for any $z$ ), the assertion turns out to be true.\\
We consider the reproducing property:
\[
  f(z) = \int_{\partial G}  K(z,w) f(w) dS_w .
\]
Since $K_z(z,w)$ is again an element of the Hardy space, we can interchange the order of the conjugated $D$-operator with the integral and we obtain
\begin{eqnarray*}
 \bar{D}_z f(z)  &=& \int_{\partial G} \bar{D}_z K(z,w) f(w) dS_w \\
&=& \int_{\partial G} \overline{K(z,w) D_z } f(w) dS_w \\
&=&\int_{\partial G}  \overline{K_{\bar{z}}(z,w)} f(w) dS_w \\
&=& \langle K_{\bar{z}}(z,w), f(w) \rangle_w .
\end{eqnarray*}
Now let
\[
  \tilde{M} := K_z(z,z) K(z,w) - K(z,z) K_{\bar{z}}(z,w) .
\]
Then
\begin{eqnarray*}\label{norm_M}
  \langle \tilde{M}, \tilde{M} \rangle_w &=& \vert K_z(z,z) K(z,w) \vert^2 + \vert K(z,z) K_{\bar{z}}(z,w) \vert^2 \\
&& - K_z(z,z) K(z,z) \langle K(z,w), K_{\bar{z}}(z,w) \rangle \\
&& - K(z,z) \langle K_{\bar{z}}(z,w), K(z,w) \rangle \overline{K_z(z,z)} \\
&=&  K_z(z,z)^2 \underbrace{K(z,z)}_{\langle K(z,w), K(z,w) \rangle}  + K(z,z)^2 \langle K_{\bar{z}}(z,w) K_{\bar{z}}(z,w) \rangle \\
&& - K_z(z,z) K(z,z) \langle K(z,w), K_{\bar{z}}(z,w) \rangle \\
&& - K(z,z) \langle K_{\bar{z}}(z,w), K(z,w) \rangle \overline{K_z(z,z)} \\
&=&  K_z(z,z)^2 K(z,z) + K(z,z)^2 K_{z\bar{z}}(z,z) \\
&& - K_z(z,z) K(z,z) K_{\bar{z}}(z,z) - K(z,z)K_z(z,z) \overline{\underbrace{K_z(z,z)}_{\in \R}} \\
&=& K(z,z)^2 K_{z\bar{z}}(z,z) - K(z,z) K_z(z,z) K_{\bar{z}}(z,z) \\
&=& K(z,z) \left(  K(z,z) K_{z\bar{z}}(z,z) - K_z(z,z) K_{\bar{z}}(z,z) \right)  \\
\end{eqnarray*}
Suppose that $\tilde{M} \equiv 0$. Then we obtain that
\begin{eqnarray*}
 0 = \langle \tilde{M}, w- z \rangle_w & = & \langle  K_z(z,z) K(z,w) - K(z,z) K_{\bar{z}}(z,w), w- z \rangle_w \\
&=& K_z(z,z) \langle K(z, w), w- z \rangle - K(z,z) \langle K_{\bar{z}}(z, w), w- z \rangle  \\
&=&  K_z(z,z)(z - z) - K(z,z) \left[  \left. \overline{D}_w (w - z) \right|_{w=z} \right] \\
&=& (n+1) K(z,z) > 0 .
\end{eqnarray*}
This is a contradiction.
\end{beweis}

With these preliminary results we are able to prove the main result of this section:

\begin{theorem}\label{thm4}
 Let $G$ be a bounded domain. Then the curvature of the Szeg\"o metric is negative on $G$.
\end{theorem}
\begin{proof}
According to the definition of the Gau{\ss}ian curvature it suffices to show that
\[
  \Delta \log K^2 (z,z) > 0.
\]
Since $K(z,z)$ is real, the product formula for the $D$-operator simplifies and we determine that
\begin{eqnarray*}
 \Delta \log K^2 (z,z) &=& \bar{D} D K(z,z) \\
 &=& \bar{D} \frac{1}{K(z,z)} K_z(z,z) \\
 &=& \dfrac{K_{\bar{z}z}(z,z) K(z,z) - K_z(z,z) K(z,z) \bar{D} }{K(z,z)^2} \\
 &+& \dfrac{K_{\bar{z}z}(z,z) K(z,z) - K_z(z,z) K_{\bar{z}}(z,z) }{K(z,z)^2} \\
\end{eqnarray*}
Here we applied in the last line that $K(z,z) \in \R$ According to
Lemma \ref{vier}, the last expression is positive, whereby the
theorem is proven.
\end{proof}

\section{The completeness result}

In order to show that the Szeg\"o metric is complete on bounded domains we will introduce first another metric which will be called the Szeg\"o-Caratheodory metric and show first that the latter one is complete.
The following two propositions are needed for the proof of the completeness of the
Szeg\"o-Carath\'{e}odory metric. Proofs of analogous statements for the case dealing with functions in several complex variables can be found for instance in \cite{Kran04}, chapter 3. Although this book deals with the function theory in several complex variables, the proofs only make use of arguments from real differential calculus in $\R^m$ and can be directly transferred due to the isometric isomorphy between $\R^{2^m}$ and $\mathcal{C}l_{0m}$ as normed vector spaces.

\begin{theorem}\label{cara_help1}
If $U$ is a domain with $C^2$ boundary, then there is an open neighbourhood $W$ of $\partial U$ such that if $z \in U \cap W$, then there is a unique point $P = P(z) \in \partial U$ which has the shortest Euclidean distance to $z$.
\end{theorem}

\begin{theorem}\label{cara_help2}
Let $U \subseteq \C$ be a bounded domain with $C^2$ boundary. Then there is an $r_0 > 0$ such that for each $P \in \partial U$ there is a disc $D(C(P),r_0)$ of radius $r_0$ which is externally tangent to $\partial U$ at $P$. There is also a disc $D(C'(P),r_0)$ which is internally tangent to $\partial U$ at $P$. This disc has furthermore the property that $D(C(P),r_0) \cap \partial U = \{P\}$ and $D(C'(P),r_0) \cap \partial U = \{P\}$.\\
\end{theorem}

We introduce
\begin{definition}

Let $G$ be a domain. The associated Hermitian metric defined by
\[
  d_C(z) := \sup \{ |\overline{D} f (z)| ; f \in H^2(G), f(z)=0, \Vert f \Vert = 1 \}
\]
is called the Szeg\"o-Carath\'{e}odory metric. Here, and in all that follows $\|\cdot\|$ denotes the norm in the Hardy space $H^2$ that is induced by the inner product defined previously.
\end{definition}

\begin{remark}
 The function $d_C$ is well-defined and positive. For a finite domain $G$ and $z_0 = \sum_{i=0}^m z^0_i e_i \in G$ we consider the function
\[
  \mathfrak{Z}_1(z) := z_1 - z_0 e_1 - z^0_1 + z^0_0 e_1 .
\]
As one can easily verify, $\mathfrak{Z}_1$ is an element of the Hardy space over $G$ and has $z_0$ as a zero. Furthermore, we have
\[
\overline{D} \mathfrak{Z}_1 = - 2 e_1 \neq 0.
\]
As now $\frac{\mathfrak{Z}_1}{\Vert \mathfrak{Z}_1 \Vert}$ falls under the definition of the set from the definition of $d_C$, we have $d_C(z) > 0$. Then again holds for each function $f$ of the Hardy space
\begin{eqnarray*}
 (\overline{D} f)(z) & = & \overline{D_z} \int_{\partial G} K(z,w) f(w) dS_w\\
& = &  \int_{\partial G} \overline{D_z} K(z,w) f(w) dS_w\\
& = & \langle \overline{D_z} K(z,\cdot), f \rangle \\
& \leq & \underbrace{\Vert \overline{D_z} K(z,\cdot) \Vert}_{< \infty} \underbrace{\Vert f \Vert}_{=1} \\
& = &\Vert \overline{D_z} K(z,\cdot) \Vert < \infty ,
\end{eqnarray*}
so in particular the supremum is really finite. Here, again H\"older's inequality has been used in the estimate. 
\end{remark}
Next we need to prove the following intertwining property of the operator $\overline{D}$:

\begin{theorem}\label{kettenregel}
Let $G \subseteq \R^m$ be a domain and $f \in C^1(G)$. Then we have that for each M\"obius transformation $ M =\left( \begin{array}{cc} a & b \\ c & d \\ \end{array} \right) $
 \[
  \left|\overline{D} (\frac{\overline{cz + d}}{|cz+d|^m} f)(M<z>)\right| = \left| \frac{\overline{cz + d}}{|cz+d|^{m+2}} \left( \overline{D}f\right)(M<z>) \right|,
\]
if $f(M<z>) = 0$.
\end{theorem}

\begin{proof}
We apply the product formula (\ref{produktregel}). As we restrict to consider functions $f$ with $f(M<z>)=0$ only, we can neglect the term $\left(\overline{D}\frac{\overline{cz+d}}{|cz+d|^m}\right) f(M<z>)$. This leaves only with terms depending on partial derivatives of $f$. If we now compare the product rules for the $D$ and the $\overline{D}$ operator, we see that the terms including $\partial_0 f$ are identical, while the terms involving $\partial_i f$, $i>0$, are identical up to a multiplication with $(-1)$. As the proposition has already been proven by Ryan in \cite{Ryan93} with the $D$ operator replacing $\overline{D}$, we can verify the equation for the $\partial_0 f$ terms using the identity
\[
 Dg + \overline{D} g = 2 \partial_0 g, \qquad g \in C^1(G).
\]
After that we can show the identity also holds for the terms involving $\partial_i f$, $i>0$,
using
\[
  Dg - \overline{D} g = \sum_{i=1}^m \partial_i g, \qquad g \in C^1(G).
\]
\end{proof}

The following theorem generalizes the distance decreasing property of the classical Carath\'{e}odory metric, compare for example with \cite{Kran04,Kran08}.

\begin{theorem}\label{metrik_ungleichung}
  For each M\"obius transformation $ M =\left( \begin{array}{cc} a & b \\ c & d \\ \end{array} \right) $ and each domain $G$ we have
\[
  \left|\frac{\overline{cz + d}}{|cz+d|^m}\right| d_C^{G}(M<z>) \leq d_C^{M<G>}(z).
\]
\end{theorem}
\begin{proof}
\begin{eqnarray*}& &
 |\frac{\overline{cz + d}}{|cz+d|^m}| d_C(M<z>)\\ & = & |\frac{{cz + d}}{|cz+d|^m}| \sup \{ |\overline{D}(f(M<z>))|; f \in H^2, f(M<z>)=0 \}\\
& = & \sup \{ |\frac{{cz + d}}{|cz+d|^m}(\overline{D}f)(M<z>))|; f \in H^2, f(M<z>)=0 \} \\
& = & \sup \{ |\overline{D}(\frac{{cz + d}}{|cz+d|^m}f(M<z>))|; f \in H^2, f(M<z>)=0 \}\\
& \leq & \sup \{ |\overline{D}f(w))|; f \in H^2(M<G>), f(w)=0 \}
\end{eqnarray*}
\end{proof}
With this tool in hand we can establish
\begin{theorem}
The Szeg\"o-Carath{\'e}odory metric is complete on each finite domain $G$ with a $C^2$-smooth boundary.
\end{theorem}

\begin{proof}
 Let $z \in G$ and $P$ be the point in $\partial G$ with minimal Euclidian distance to $z$. Let furthermore $\delta$ be the Euclidean distance of $z$ to the boundary. According to (\ref{cara_help1}), (\ref{cara_help2}) there exists a $r_0 > 0$ and $C(P) \in \R^{m+1} \setminus G$ so that $D(C(P),r_0) \cap \partial G = \{P\}$.\\
We now define the maps
\[
  i_P: z \mapsto \frac{r_0 \overline{z} - \overline{C(P)}}{|z - C(P)|^2} + C(P)
\]
and
\[
  j_P: z \mapsto \frac{1}{r_0} \left( z - C(P) \right) .
\]
Both $i_P$ and $j_P$ are apparently M\"obius tranformations.\\
For their composition it holds
\begin{eqnarray*}
 (j_P \circ i_P)(z) & = & \frac{1}{r_0} \left( r_0^2 \frac{\overline{z}-\overline{C(P)}}{|z - C(P)|^2} + C(P) - C(P) \right)  \\
 & = & r_0 \frac{\overline{z}-\overline{C(P)}}{|z - C(P)|^2} \\
 & = & r_0 \left( z - C(P) \right)^{-1}
\end{eqnarray*}
 From the choice of $z$ and $C(P)$ we gain for the absolute value
\[
|(j_P \circ i_P)(z)| = r_0 \frac{1}{\delta + r_0} = 1 - \frac{\delta}{\delta + r_0}.
\]
We now consider the automorphic factor belonging to the M\"obius tranformation $j_P \circ i_P$ according to (\ref{kettenregel}):
\[
  \left| r_0 \frac{\overline{z}-\overline{C(P)}}{|z - C(P)|^{m+1}} \right| = r_0 \frac{1}{(\delta + r_0)^{m}}
\]
Applying now Theorem \ref{metrik_ungleichung} leads to
\begin{equation}\label{2te_ungl}
 d_C^G(z) \geq \left| r_0 \frac{\overline{z}-\overline{C(P)}}{|z - C(P)|^{m+1}} \right| d_C^{\mathbb{D}}((j_P \circ i_P)(z)) = \left| r_0 \frac{\overline{z}-\overline{C(P)}}{|z - C(P)|^{m+1}} \right| d_C^{\mathbb{D}}(\frac{r_0}{z - C(P)}).
\end{equation}
We thereby denote the unit ball in $\R^{m+1}$ with $\mathbb{D}$.\\
To estimate the Szeg\"o-Carath{\'e}odory metric on the unit ball, we define a test function.\\
Let $\mathfrak{K}$ be the Cauchy kernel and $w_0 \in \R^{m+1} \setminus G$ so that $|z-w_0|=2\delta$. then it holds that
\[
  D\mathfrak{K}( \cdot - w_0) = 0
\]
on $G$ and $\mathfrak{K}$ is welldefined on $G$.\\
We now set
\[
  \mathfrak{K}_2 (z) := \overline{D_z} \left( \mathfrak{K} ( \cdot - w_0)  \right) - \overline{D_z} \left( \mathfrak{K} ( \frac{r_0}{\delta + r_0} - w_0)  \right) .
\]
By construction we have $\mathfrak{K}_2 \in H^2(G, \mathcal{C}_{0m})$ and $\mathfrak{K}_2 (\frac{r_0}{\delta + r_0} - w_0) = 0$.\\
Applying this property leads to
\[
 d_C^{\mathbb{D}}(\frac{r_0}{\delta + r_0}) \geq \left| \left( \overline{D} \mathfrak{K}_2 \right)  (\frac{r_0}{\delta + r_0}) \right|.
\]
We obtain that
\begin{eqnarray}
 \overline{D}^2 & = & (\partial_0 + \overrightarrow{D})^2 \\
 & = & \partial_0^2 - \partial_0 \overrightarrow{D} - \overrightarrow{D} \partial_0 - \sum_{i=1}^{m-1} \partial^2_i,
\end{eqnarray}
so
\[
  \Delta \mathfrak{K}_2 = 0 ,
\]
and therefore
\[
 - \sum_{i=1}^{m} \partial^2_i \mathfrak{K}_2 = \partial_0^2 \mathfrak{K}_2
\]
As $\mathfrak{K}_2 \in C^{\infty}(G)$, we get
\[
 \partial_0 \overrightarrow{D} \mathfrak{K}_2 = \overrightarrow{D} \partial_0 \mathfrak{K}_2 .
\]
Because of $D \mathfrak{K}_2 = 0$, we have
\[
 \overrightarrow{D} \mathfrak{K}_2 = - \partial_0 \mathfrak{K}_2.
\]
 From this property we can deduce that
\begin{eqnarray*}
 \overline{D}^2 \mathfrak{K}_2 (z) & = & 4 \partial_0^2 \mathfrak{K}_2 (z) \\
& = & 4 \partial_0 \frac{|z-w_0|^{2} - (m+1) (\overline{z-w_0})(z_0 - {w_0}_0)}{|z-w_0|^{m+3}} \\
& = & 4 \Bigg[ \frac{-(m+1)(z_0 - {w_0}_0)}{|z-w_0|^{m+3}}\\ & & \quad\quad - (m+1)\frac{-\overrightarrow{z-w_0}}{|z-w_0|^{m+3}} + (m+3)\frac{\overline{z-w_0}(z_0 - {w_0}_0)^2}{|z-w_0|^{m+5}} \Bigg] \\
& = & 4 \left[ -(m+1) \frac{ \overline{z-w_0} }{|z-w_0|^{m+3}} \right].
\end{eqnarray*}
As only small values of $\delta$ are relevant in our case, we can assume that $\delta<1$ without loss of generality. Under this assumption, we can estimate
\[
|\overline{D}^2 \mathfrak{K}_2 (z)|
\]
from below through
\[
 C_1 \frac{1}{|z-w_0|^{m+2}} = C_1 \frac{1}{(2\delta)^{m+2}}
\]
with a constant $C_1 > 0$.\\
Furthermore, we can estimate $\Vert \overline{D} \mathfrak{K}_2 (z) \Vert$ from below through $\tilde{C}_1 \frac{1}{(2 \delta)^{m+1}}$.\\
Together with (\ref{2te_ungl}) we then have
\begin{equation}\label{scm_endabschaetzung}
d_C^G(z) \geq r_0 \frac{1}{(\delta + r_0)^{m}} C_1 \frac{1}{(2\delta)^{m+2}} \frac{1}{\tilde{C}_1 \frac{1}{(2 \delta)^{m+1}}} \geq C_2(r_0) \frac{1}{\delta}
\end{equation}
with a number $C_2 > 0$ which only depends on $r_0$ .\\
Let now $E$  be a set which is bounded in the Szeg\"o-Carath{\'e}odory metric, then
\[
  d_C^G(z) < M
\]
for all $z \in E$ for an $M>0$. According to (\ref{scm_endabschaetzung}) this yields to
\[
  \delta \geq C_3(r_0) M.
\]
So there exists a positive minimal distance of $E$ to $\partial G$. Due to the boundedness of $G$ $E$ then is relatively compact. Thus, the Szeg\"o-Carath{\'e}odory metric is complete on $G$.
\end{proof}

To show the completeness of the Szeg\"o metric, we first prove an inequality between it and the Szeg\"o-Carath\'{e}odory metric on a bounded domain, which will be essential for the final proof. The proof of the completeness of the Bergman metric in the theory of several complex variables makes use of a similar inequality, which can be found in \cite{Koba62,Hahn76,Hahn77,Hahn78} and \cite{Kran01}.

\begin{theorem}\label{relation_s_c_metriken}
 Let $C_G$ be the Szeg\"o-Carath{\'e}odory metric of a domain $G$. Then
\begin{enumerate}
 \item[a)] $ds_G \geq C_G$
 \item[b)] The metric $ds_G^*$ given by $ \Delta \log S(z,z) |dz|$ satisfies the inequality
\[
   ds_G^* \geq C_G.
\]
\end{enumerate}

\end{theorem}


\begin{proof}
Let furthermore be $f'(z):=\overline{D} f$.\\
Let $z \in G$ $f \in H^2$ with $f(z)=0$ and $\| f \| \leq 1$. We further define
\[
  M:=K(z,z)K_{\bar{z}}(z, \cdot) - K_{z}(z, z)K(z, \cdot) .
\]
It holds that
\[
  0 =  f(z) = \langle f, K(z, \cdot)\rangle .
\]
 From this property we may infer that
 \begin{eqnarray*}
  |f'(z)| &=& \frac{1}{K(z,z)} \langle f, K(z,z)K_{\bar{z}}(z, \cdot) - K_{z}(z, z)K(z, \cdot)\rangle \\
 &\leq & \underbrace{\| f \|}_{\leq 1}  \frac{1}{K(z,z)} \| M \| \\
&\stackrel{(\ref{norm_M})}{\leq}& \frac{\sqrt{K(z,z)(K(z,z)K_{z\bar{z}}(z,z) - K_z(z,z)K_{\bar{z}}(z,z)) }}{K(z,z)} \\
&=& \frac{\sqrt{(K(z,z)K_{z\bar{z}}(z,z)) - K_z(z,z)K_{\bar{z}}(z,z) }}{\sqrt{K(z,z)}} \\
&=&\frac{\sqrt{M'(z)}}{\sqrt{K(z,z)}} \\
&=& \sqrt{\Delta \log K^2 (z,z)}.
\end{eqnarray*}
In the second line, we applied H\"older's inequality. In the last line we notice that the argument of the square root is really positive. This is a consequence of Theorem~\ref{thm4}. 

Next, each
function $f \in H^2(G, \mathcal{C}l_{0m} )$ with $\Vert f \Vert
\leq 1$ satisfies
\begin{eqnarray*}
 |f(z)| &=& |\langle f, K(z, \cdot) \rangle| \\
& \leq & \Vert f \Vert \Vert K(z, \cdot) \Vert \\
&\leq& \sqrt{\langle K(z, \cdot), K(z, \cdot) \rangle}\\
&=& \sqrt{K(z,z)}.\\
\end{eqnarray*}
Again, in the second line we applied H\"older's inequality. Notice that we require the result of lemma \ref{vier} for the root to be well-defined.\\
Furthermore, we have with $g(w):= \frac{M(w)}{\Vert M \Vert}$ that
$\Vert g \Vert \leq 1$ and
 \begin{eqnarray*}
  d_C(z) &\geq & |g'(z)| \\
&=& |\frac{M'(z)}{\Vert M \Vert}| \\
&=& \frac{\Vert M \Vert^2}{\Vert M \Vert K(z,z)} \\
&=& \frac{\Vert M \Vert}{K(z,z)}
 \end{eqnarray*}
A simple calculation shows that $\Vert M \Vert = \Vert M' \Vert$.\\
 From this equality it follows as a consequence of $DM'=0$  that
\[
  \frac{M'(z)}{\Vert M \Vert} \leq \sqrt{K(z,z)},
\]
and thus
\[
   \Vert M \Vert \frac{1}{K(z,z)} \leq \sqrt{K(z,z)}.
\]
So, powers of both metrics can be estimated from below through the Szeg\"o-Carath{\'e}odory metric. But from this property we may draw the conclusion that each set $E$ which is bounded in these metrices is also bounded in the Szeg\"o-Carath{\'e}odory metric. As this metric is complete, $E$ must be relatively compact in $G$. Thus both these metrices are also complete and the propostion is proven.
\end{proof}

 From the completeness of the Szeg\"o-Carath{\'e}odory metric we can finally establish our main result of this section:

\begin{theorem}
 The Szeg\"o metric is complete.
\end{theorem}
\begin{proof}
 In regard of Theorem \ref{relation_s_c_metriken}, we have for all $z, w \in G$
\[
  d_C(z,w) \leq d_S(z, w).
\]
So each $d_S$-bounded set $H \subset G$ is simultaneously
$d_C$-bounded. As the Szeg\"o-Carath{\'e}odory metric is complete,
$H$ must be relatively compact in $G$. With this, each
$d_S$-bounded set is also relatively compact in $G$ and $d_S$ thus
complete.
\end{proof}

Perspectives:
We have already mentioned that in the theory of several complex variables there exists a strong connection between the completeness of the metric and the smoothness of the boundary of the domain. So it is an obvious point of investigation for further works if and how these propositions can be transferred to the Clifford case, as our hitherto existing results require a strong smoothness ($C^2$-boundary).

\par\medskip\par Acknowledgements. The second author gratefully acknowledges the fruitful und helpful discussions with Dr. Denis Constales from Ghent University on the topic of this paper which lead to the successful development of it.

\end{document}